\documentclass[12pt,leqno]{article}

\usepackage{amsmath,amssymb,amsthm}
\usepackage[all]{xy}

\makeatletter

\newtheorem*{theorem-main}{Theorem~\ref{main.thm}}

\newtheorem{theorem}{Theorem}[section]
\newtheorem{proposition}[theorem]{Proposition}
\newtheorem{lemma}[theorem]{Lemma}
\newtheorem{corollary}[theorem]{Corollary}

\theoremstyle{definition}

\newtheorem{definition}[theorem]{Definition}

\theoremstyle{remark}

\theoremstyle{remark}

\def\({{\rm (}}
\def\){{\rm )}}

\let\Mathrm\operator@font
\let\Cal\mathcal
\let\Bbb\mathbb
\newcommand{\fm}{\ensuremath{\mathfrak m}}

\def\standop#1{\mathop{\Mathrm #1}\nolimits}
\def\difstop#1#2{\expandafter\def\csname #1\endcsname{\standop{#2}}}
\def\defstop#1{\difstop{#1}{#1}}

\defstop{AB}
\defstop{ann}
\defstop{Ass}
\defstop{Add}
\defstop{Alt}
\defstop{Ass}
\defstop{Assh}
\defstop{add}

\defstop{CMFI}
\defstop{codim}
\defstop{Coh}
\defstop{coht}
\defstop{Coker}
\defstop{Cone}
\defstop{Cl}
\defstop{Cox}
\defstop{cosk}
\defstop{cd}

\defstop{depth}
\defstop{Der}
\defstop{Div}
\defstop{div}

\def\da{^\dagger}
\def\dda{^\ddagger}
\def\ddd{^{\dagger\ddagger}}

\defstop{EM}
\defstop{embdim}
\defstop{End}
\defstop{ev}
\defstop{Ext}

\defstop{Flat}
\defstop{Func}
\defstop{Fpqc}

\defstop{GD}
\defstop{Good}
\defstop{Gal}

\difstop{height}{ht}
\defstop{Hom}

\difstop{Image}{Im}
\defstop{ind}
\defstop{idim}

\defstop{Ker}

\defstop{Lch}
\defstop{length}
\defstop{Lin}
\defstop{Lqc}
\defstop{lqc}
\defstop{LQ}
\defstop{LM}
\defstop{Lie}
\defstop{Loc}
\def\Lop{_{\Lambda\op}}

\defstop{Mat}
\defstop{Max}
\defstop{Min}
\defstop{Mod}
\difstop{md}{mod}
\defstop{Mor}
\defstop{MCM}
\defstop{Map}
\difstop{mred}{red}

\defstop{Nerve}
\defstop{NonSFR}
\defstop{NonCMFI}
\defstop{NonNor}
\defstop{Nor}

\defstop{ob}
\defstop{Ob}
\def\op{^{\standop{op}}}

\defstop{PA}
\defstop{PM}
\defstop{PR}
\defstop{Proj}
\defstop{Prin}
\defstop{Pic}
\defstop{Push}

\defstop{Qch}
\defstop{qch}

\defstop{rad}
\defstop{rank}

\defstop{res}
\defstop{Reg}
\defstop{Ref}

\defstop{Spec}
\defstop{supp}
\defstop{Supp}
\defstop{Sym}
\defstop{Sing}
\defstop{SFR}
\defstop{Soc}
\defstop{Sh}

\defstop{syz}
\defstop{Syz}
\defstop{SPH}

\difstop{sm}{sum}

\difstop{tdeg}{trans.deg}

\defstop{Tor}
\difstop{trace}{tr}
\defstop{TF}
\defstop{tf}
\defstop{Tr}

\defstop{up}
\defstop{UP}

\defstop{Zar}

\def\C{\Cal C}

\def\M{\Cal M}
\def\N{\Cal N}
\def\O{\Cal O}

\def\Q{\Cal Q}

\def\V{\Cal V}

\def\fm{\mathfrak{m}}

\def\uHom{\mathop{\text{\underline{$\Mathrm Hom$}}}\nolimits}
\def\uEnd{\mathop{\text{\underline{$\Mathrm End$}}}\nolimits}

\def\sdarrow#1{\downarrow\hbox to 0pt{\scriptsize$#1$\hss}}
\def\suarrow#1{\uparrow\hbox to 0pt{\scriptsize$#1$\hss}}
\def\ssearrow#1{\searrow\hbox to 0pt{\scriptsize$#1$\hss}}

\def\section{\@startsection{section}{1}{\z@ }%
  {-3.5ex plus -1ex minus -.2ex}{2.3ex plus .2ex}{\bf }}

\long\def\refname{\par\kern -3ex
  \begin{center}\rm R\sc{eferences}\end{center}\par\kern 
  -2ex}

\def\@seccntformat#1{\csname the#1\endcsname.\quad}

\def\@@@sect#1#2#3#4#5#6[#7]#8{%
  \ifnum #2>\c@secnumdepth 
  \def \@svsec {}\else \refstepcounter {#1}%
  \def\@svsec{}
  \fi 
  \@tempskipa #5\relax 
  \ifdim \@tempskipa >\z@ 
  \begingroup #6\relax \@hangfrom {\hskip #3\relax 
    \@svsec}{\interlinepenalty \@M #8\par }\endgroup 
  \csname #1mark\endcsname {#7}
  \else 
  \def \@svsechd {#6\hskip #3\@svsec #8\csname #1mark\endcsname {#7}}
  \fi \@xsect {#5}}

\def\@@@startsection#1#2#3#4#5#6{%
  \if@noskipsec \leavevmode \fi \par \@tempskipa #4\relax \@afterindenttrue 
  \ifdim \@tempskipa <\z@ \@tempskipa -\@tempskipa \@afterindentfalse 
  \fi \if@nobreak \everypar {}\else \addpenalty {\@secpenalty }\addvspace 
  {\@tempskipa }\fi \@ifstar {\@ssect {#3}{#4}{#5}{#6}}{\@dblarg 
    {\@@@sect {#1}{#2}{#3}{#4}{#5}{#6}}}}

\def\theparagraph{\thesection.\arabic{paragraph}}
\def\aparagraph{\@@@startsection{paragraph}{2}{\z@ }%
  {1.75ex plus .2ex minus .15ex}{-1em}{\bf(\theparagraph) } }
\def\paragraph{\@@@startsection{paragraph}{2}{\z@ }%
  {1.75ex plus .2ex minus .15ex}{-1em}{}{\bf(\theparagraph)} }

\let\c@theorem\c@paragraph

\advance\textheight 10pt
\advance\textwidth 6pt
\title{Higher-dimensional absolute versions
  of symmetric, Frobenius, and quasi-Frobenius algebras}
\author{M{\sc itsuyasu} H{\sc ashimoto}\thanks{This work was supported by JSPS KAKENHI Grant Number 26400045.}}
\date{\normalsize
  Department of Mathematics, Okayama University\\
  Okayama 700--8530, JAPAN\\
  {\small \tt mh@okayama-u.ac.jp}
}

\begin{document}

\maketitle

\begin{abstract}
  In this paper, we define and discuss higher-dimensional and
  absolute versions of symmetric, Frobenius, and quasi-Frobenius
  algebras.
  In particular, we compare these with the relative notions defined
  by Scheja and Storch.
  We also prove the validity of codimension two-argument for modules over
  a coherent sheaf of algebras with a $2$-canonical module,
  generalizing a result of the author.
\end{abstract}


\section{Introduction}

\paragraph
Let $(R,\fm)$ be a semilocal Noetherian commutative ring, and $\Lambda$ a
module-finite $R$-algebra.
In \cite{Hashimoto2}, we defined the canonical module $K_\Lambda$ of $\Lambda$.
The purpose of this paper is two fold, each of which is deeply related to $K_\Lambda$.

\paragraph\label{first-part.par}
In the first part, we define and discuss higher-dimensional and
absolute notions of 
symmetric, Frobenius, and quasi-Frobenius algebras and their non-Cohen--Macaulay versions.
In commutative algebra, the non-Cohen--Macaulay version of Gorenstein ring is known as
quasi-Gorenstein rings.
What we discuss here is a non-commutative version of such rings.
Scheja and Storch \cite{SS} discussed a relative notion, and our definition is
absolute in the sense that it depends only on $\Lambda$ and
is independent of the choice of $R$.
If $R$ is local, our quasi-Frobenius property agrees with
Gorensteinness discussed by Goto and Nishida \cite{GN},
see Proposition~\ref{GN.prop} and Corollary~\ref{GN.cor}.

\paragraph\label{second-part.par}
In the second part, we show that the codimension-two argument using the
existence of $2$-canonical modules in \cite{Hashimoto12} is still valid in
non-commutative settings.
For the definition of an $n$-canonical module, see (\ref{n-can.par}).
Codimension-two argument, which states (roughly speaking) that removing a closed subset
of codimension two or more does not change the
category of coherent sheaves which satisfy Serre's $(S_2')$ condition, 
is sometimes used in algebraic geometry, commutative algebra and invariant theory.
For example, information on the canonical sheaf and the class group is
retained when we remove the singular locus of a normal variety over an algebraically closed
field, and then these objects
are respectively grasped as the top exterior power of the cotangent
bundle and the Picard group of a smooth variety.
In \cite{Hashimoto12}, almost principal bundles are studied.
They are principal bundles after removing closed subsets of codimension two or more.

We prove the following.
Let $X$ be a locally Noetherian scheme, $U$ an open subset of $X$ such that
$\codim_X(X\setminus U)\geq 2$.
Let $i:U\rightarrow X$ be the inclusion.
Let $\Lambda$ be a coherent $\O_X$-algebra.
If $X$ possesses a $2$-canonical module $\omega$, then the inverse image $i^*$ induces
the equivalence between the category of coherent right $\Lambda$-modules which
satisfy the $(S_2')$ condition and the category of coherent right $i^*\Lambda$-modules
which satisfy the $(S_2')$ condition.
The quasi-inverse is given by the direct image $i_*$.
What was proved in \cite{Hashimoto12} was the case that $\Lambda=\O_X$.
If, moreover, $\omega=\O_X$ (that is to say, $X$ satisfy the
$(S_2)$ and $(G_1)$ condition), then the assertion has been well-known,
see \cite{Hartshorne4}.

\paragraph
$2$-canonical modules are ubiquitous in algebraic geometry.
If $\Bbb I$ is a dualizing complex of a Noetherian scheme $X$, then
the lowest non-vanishing cohomology group of $\Bbb I$ is semicanonical.
A rank-one reflexive sheaf over a normal variety is $2$-canonical.

\paragraph
Section~\ref{preliminaries.sec} is for preliminaries.
Section~\ref{Frobenius.sec} is devoted to the discussion of the first theme mentioned
in the paragraph (\ref{first-part.par}), while 
Section~\ref{codim-two.sec} is for the second theme mentioned in (\ref{second-part.par}).

\paragraph
Acknowledgments:
Special thanks are due to Professor
Osamu Iyama for valuable advice and
discussion.

The essential part of this paper has first appeared as \cite[sections~9--10]{Hashimoto}.
When it is published as \cite{Hashimoto2}, they have been removed after the
requirement to shorten the paper (also, the title has been changed slightly).
Here we revive them as an independent paper.

\section{Preliminaries}\label{preliminaries.sec}

\paragraph
Throughout this paper, $R$ denotes a Noetherian commutative ring.
For a module-finite
$R$-algebra $\Lambda$, a $\Lambda$-module means a left $\Lambda$-module.
$\Lambda\op$ denotes the opposite algebra of $\Lambda$, and thus a
$\Lambda\op$-module is identified with a right $\Lambda$-module.
A $\Lambda$-bimodule means a $\Lambda\otimes_R \Lambda\op$-module.
The category of finite $\Lambda$-modules is denoted by $\Lambda\md$.
The category $\Lambda\op\md$ is also denoted by $\md\Lambda$.

\paragraph Let $(R,\fm)$ be semilocal and $\Lambda$ be a module-finite $R$-algebra.
For an $R$-module $M$, the $\fm$-adic completion of $M$ is denoted by $\hat M$.
For a finite $\Lambda$-module $M$, by $\dim M$ or $\dim_\Lambda M$ we mean
$\dim_R M$, which is independent of the choice of $R$.
By $\depth M$ or $\depth_\Lambda M$ we mean $\depth_R(\fm,M)$, which is independent of
$R$.
We say that $M$ is globally Cohen--Macaulay (GCM for short) if
$\dim M=\depth M$.
We say that $M$ is globally maximal Cohen--Macaulay (GMCM for short) if
$\dim\Lambda=\depth M$.
If $R$ happens to be local, then $M$ is GCM (resp.\ GMCM) if and only if
$M$ is Cohen--Macaulay (resp.\ maximal Cohen--Macaulay) as an $R$-module.

\paragraph
For $M\in\Lambda\md$, we say that $M$ satisfies $(S_n')^{\Lambda,R}$,
$(S_n')^R$ or $(S_n')$ if
$\depth_{R_P}M_P\geq \min(n,\height_R P)$ for every $P\in\Spec R$
(this notion depends on $R$).

\paragraph
Let $X$ be a locally Noetherian scheme and $\Lambda$ a coherent $\O_X$-algebra.
For a coherent $\Lambda$-module $\M$, we say that $\M$ satisfies $(S_n')$
or $(S_n')^{\Lambda,X}$, or sometimes $\M\in (S_n')^{\Lambda,X}$, if
$\depth_{\O_{X,x}}\M_x\geq \min(n,\dim \O_{X,x})$ for every $x\in X$.

\paragraph\label{canonical-complete-def.par}
Assume that $(R,\fm)$ is complete semilocal, and $\Lambda\neq 0$ a module-finite
$R$-algebra.
Let $\Bbb I$ be a normalized dualizing complex of $R$.
The lowest non-vanishing cohomology group $\Ext^{-s}_R(\Lambda,\Bbb I)$
($\Ext^i_R(\Lambda,\Bbb I)=0$ for $i<-s$) is denoted by $K_\Lambda$, and is called the
{\em canonical module} of $\Lambda$.
If $\Lambda=0$, then we define that $K_\Lambda=0$.
For basics on the canonical modules, we refer the reader to \cite{Hashimoto2}.
Note that $K_\Lambda$ depends only on $\Lambda$, and is independent of $R$.

\paragraph
Assume that $(R,\fm)$ is semilocal which may not be complete.
We say that a finitely generated $\Lambda$-bimodule $K$ is
a {\em canonical module} of $\Lambda$
if $\hat K$ is isomorphic to
the canonical module $K_{\hat \Lambda}$ as a 
$\hat \Lambda$-bimodule.
It is unique up to isomorphisms, and denoted by $K_\Lambda$.
We say that $K \in \md \Lambda$ is a right (resp.\ left)
canonical module of $\Lambda$
if $\hat K$ is isomorphic to $K_{\hat \Lambda}$ in $\md \hat \Lambda$
(resp.\ $\hat\Lambda\md$).
If $K_\Lambda$ exists, then $K$ is a right canonical module if and only if
$K\cong K_\Lambda$ in $\md\Lambda$.

\paragraph
We say that $\omega$
is an $R$-semicanonical right $\Lambda$-module if for any $P\in\Spec R$,
$R_P \otimes_R \omega$ is
the right canonical module
$R_P\otimes_R \Lambda$ 
for any $P\in\supp_R \omega$.

\paragraph\label{n-can.par}
Let $C\in\md \Lambda$.
We say that $C$ is an
$n$-canonical right $\Lambda$-module over $R$ if $C\in (S_n')^R$,
and for each $P\in\Spec R$ with $\height P<n$, we have that $C_P$ is an
$R_P$-semicanonical right $\Lambda_P$-module.

\section{Symmetric and Frobenius algebras}\label{Frobenius.sec}

\paragraph Let $(R,\fm)$ be a Noetherian semilocal ring, and $\Lambda$ a
module-finite $R$-algebra.
Let $K_\Lambda$ denote the canonical module of $\Lambda$, see \cite{Hashimoto2}.

We say that $\Lambda$ is {\em quasi-symmetric}
if $\Lambda$ is the canonical module of $\Lambda$.
That is, $\Lambda\cong K_\Lambda$ as $\Lambda$-bimodules.
It is called {\em symmetric} if it is quasi-symmetric and GCM.
Note that $\Lambda$ is quasi-symmetric (resp.\ symmetric) if and only
if $\hat \Lambda$ is so.
Note also that quasi-symmetric and symmetric are absolute notion, and
is independent of the choice of $R$ in the sense that the definition
does not change when we replace $R$ by the center of $\Lambda$,
because $K_\Lambda$ is independent of the choice of $R$.

\paragraph
For (non-semilocal) Noetherian ring $R$,
we say that $\Lambda$ is locally quasi-symmetric
(resp.\ locally symmetric) over $R$ if for any $P\in \Spec R$,
$\Lambda_P$ is a quasi-symmetric (resp.\ symmetric) $R_P$-algebra.
This is equivalent to say that for any maximal ideal $\fm $ of $R$,
$\Lambda_\fm$ is quasi-symmetric (resp.\ symmetric), see
\cite[(7.6)]{Hashimoto2}.

In the case that $(R,\fm)$ is semilocal,
$\Lambda$ is 
locally quasi-symmetric (resp.\ locally symmetric) over $R$ if
it is quasi-symmetric (resp.\ symmetric), but the converse is
not true in general.

\begin{lemma}\label{pseudo-Frobenius.lem}
  Let $(R,\fm)$ be a Noetherian
  semilocal ring, and $\Lambda$ a module-finite
  $R$-algebra.
  Then the following are equivalent.
  \begin{enumerate}
  \item[\bf 1] $\Lambda_\Lambda$ is the right canonical module of $\Lambda$.
  \item[\bf 2] ${}_\Lambda\Lambda$ is the left canonical module of $\Lambda$.
  \end{enumerate}
\end{lemma}

\begin{proof}
  We may assume that $R$ is complete.
  Then replacing $R$ by a Noether normalization of
  $R/\ann_R \Lambda$, we may assume
  that $R$ is regular and $\Lambda$ is a faithful $R$-module.

  We prove {\bf 1$\Rightarrow$\bf 2}.
  By \cite[Lemma~5.10]{Hashimoto2}, $K_\Lambda$ satisfies $(S_2')^R$.
  By assumption, $\Lambda_\Lambda$ satisfies $(S_2')^R$.
  As $R$ is regular and $\dim R=\dim \Lambda$, $K_\Lambda=\Lambda^*
  =\Hom_R(\Lambda,R)$.
  So we get an $R$-linear map
  \[
  \varphi:\Lambda\otimes_R\Lambda \rightarrow R
  \]
  such that $\varphi(ab\otimes c)=\varphi(a\otimes bc)$ and that
  the induced map $h:\Lambda\rightarrow \Lambda^*$ given by $h(a)(c)=\varphi(a\otimes
  c)$ is an isomorphism (in $\md \Lambda$).
  Now $\varphi$ induces a homomorphism $h':\Lambda\rightarrow \Lambda^*$ in
  $\Lambda\md$ given by $h'(c)(a)=\varphi(a\otimes c)$.
  To verify that this is an isomorphism, as $\Lambda$ and $\Lambda^*$ are
  reflexive $R$-modules, we may localize at a prime $P$ of $R$ of height at most
  one, and then take a
  completion, and hence we may further assume that $\dim R\leq 1$.
  Then $\Lambda$ is a finite free $R$-module, and the matrices of $h$ and $h'$ are
  transpose each other.
  As the matrix of $h$ is invertible, so is that of $h'$, and $h'$ is an isomorphism.

  {\bf 2$\Rightarrow$1} follows from {\bf 1$\Rightarrow$2},
  considering the opposite ring.
\end{proof}

\begin{definition}
  Let $(R,\fm)$ be semilocal.
  We say that $\Lambda$ is a {\em pseudo-Frobenius $R$-algebra}
  if the equivalent conditions of Lemma~\ref{pseudo-Frobenius.lem} are
  satisfied.
  If $\Lambda$ is GCM in addition, then it is called a
  {\em Frobenius $R$-algebra}.
  Note that these definitions
  are independent of the choice of $R$.
  Moreover, $\Lambda$ is pseudo-Frobenius (resp.\ Frobenius) if and only
  if $\hat\Lambda$ is so.
  For a general $R$, we say that $\Lambda$ is locally pseudo-Frobenius
  (resp.\ locally Frobenius) over $R$
  if $\Lambda_P$ is pseudo-Frobenius (resp.\ Frobenius) for $P\in\Spec R$.
\end{definition}

\begin{lemma}\label{quasi-Frobenius.lem}
  Let $(R,\fm)$ be semilocal.
  Then the following are equivalent.
  \begin{enumerate}
  \item[\bf 1] 
    $(K_{\hat\Lambda})_{\hat\Lambda}$ is projective in $\md\hat\Lambda$.
  \item[\bf 2] 
    ${}_{\hat\Lambda} (K_{\hat \Lambda})$ 
    is projective in $\hat\Lambda\md$,
  \end{enumerate}
  where $\hat ?$ denotes the $\fm$-adic completion.
\end{lemma}

\begin{proof}
  We may assume that $(R,\fm,k)$ is complete regular local and
  $\Lambda$ is a faithful $R$-module.
  Let $\bar ?$ denote the functor $k\otimes_R ?$.
  Then $\bar \Lambda$ is a finite dimensional $k$-algebra.
  So $\md\bar\Lambda$ and $\bar\Lambda\md$ have the same number of
  simple modules, say $n$.
  An indecomposable projective module in $\md\Lambda$ is nothing but
  the projective cover of a simple module in $\md\bar\Lambda$.
  So $\md\Lambda$ and $\Lambda\md$ have $n$ indecomposable projectives.
  Now $\Hom_R(?,R)$ is an equivalence between $\add (K_\Lambda)_\Lambda$ and
  $\add {}_\Lambda \Lambda$.
  It is also an equivalence between $\add {}_\Lambda (K_\Lambda)$ and
  $\add\Lambda_\Lambda$.
  So both $\add(K_\Lambda)_\Lambda$ and $\add{}_\Lambda (K_\Lambda)$ also have
  $n$ indecomposables.
  So {\bf 1} is equivalent to $\add (K_\Lambda)_\Lambda = \add \Lambda_\Lambda$.
  {\bf 2} is equivalent to $\add {}_\Lambda (K_\Lambda) = \add {}_\Lambda\Lambda$.
  So {\bf 1$\Leftrightarrow $2} is proved simply applying
  the duality $\Hom_R(?,R)$.
\end{proof}

\paragraph
Let $(R,\fm)$ be semilocal.
If the equivalent conditions in Lemma~\ref{quasi-Frobenius.lem} are
satisfied, then we say that $\Lambda$ is {\em pseudo-quasi-Frobenius}.
If it is GCM in addition, then we say that it is
{\em quasi-Frobenius}.
These definitions are independent of the choice of $R$.
Note that $\Lambda$ is pseudo-quasi-Frobenius (resp.\ quasi-Frobenius)
if and only if $\hat\Lambda$ is so.

\begin{proposition}\label{GN.prop}
  Let $(R,\fm)$ be semilocal.
  Then the following are equivalent.
  \begin{enumerate}
  \item[\bf 1] $\Lambda$ is quasi-Frobenius.
  \item[\bf 2] $\Lambda$ is GCM, and $\dim\Lambda=\idim {}_\Lambda\Lambda$,
    where $\idim$ denotes the injective dimension.
  \item[\bf 3] $\Lambda$ is GCM, and $\dim\Lambda=\idim \Lambda_\Lambda$.
  \end{enumerate}
\end{proposition}

\begin{proof}
  {\bf 1$\Rightarrow $2}.
  By definition, $\Lambda$ is GCM.
  To prove that $\dim\Lambda=\idim {}_\Lambda\Lambda$, we may assume that
  $R$ is local.
  Then by \cite[(3.5)]{GN}, we may assume that $R$ is complete.
  Replacing $R$ by the Noetherian normalization of
  $R/\ann_R \Lambda$, we may
  assume that $R$ is a complete regular local ring of dimension $d$, and
  $\Lambda$ its maximal Cohen--Macaulay (that is, finite free) module.
  As $\add {}_\Lambda \Lambda = \add {}_\Lambda (K_\Lambda)$ by
  the proof of Lemma~\ref{quasi-Frobenius.lem}, it suffices to
  prove $\idim {}_\Lambda (K_\Lambda)=d$.
  Let $\Bbb I_R$ be the minimal injective resolution of the
  $R$-module $R$.
  Then $\Bbb J=\Hom_R(\Lambda,\Bbb I_R)$
  is an injective resolution
  of $K_\Lambda=\Hom_R(\Lambda,R)$ as a left $\Lambda$-module.
  As the length of $\Bbb J$ is $d$ and
  \[
  \Ext^d_\Lambda(\Lambda/\fm\Lambda,K_\Lambda)\cong
  \Ext^d_R(\Lambda/\fm\Lambda,R)\neq 0,
  \]
  we have that $\idim{}_\Lambda (K_\Lambda)=d$.

  {\bf 2$\Rightarrow$1}.
  We may assume that $R$ is complete regular local and
  $\Lambda$ is maximal Cohen--Macaulay.
  By \cite[(3.6)]{GN}, we may further assume that $R$ is a field.
  Then ${}_\Lambda \Lambda$ is injective.
  So $(K_\Lambda)_\Lambda=\Hom_R(\Lambda,R)$ is projective,
  and $\Lambda$ is quasi-Frobenius, see \cite[(IV.3.7)]{SY}.

  {\bf 1$\Leftrightarrow$3} is proved similarly.
\end{proof}

\begin{corollary}\label{GN.cor}
  Let $R$ be arbitrary.
  Then the following are equivalent.
  \begin{enumerate}
  \item[\bf 1] For any $P\in\Spec R$, $\Lambda_P$ is quasi-Frobenius.
  \item[\bf 2] For any maximal ideal $\fm$ of $R$, $\Lambda_\fm$ is quasi-Frobenius.
  \item[\bf 3] $\Lambda$ is a Gorenstein $R$-algebra in the sense that
    $\Lambda$ is a Cohen--Macaulay $R$-module, and
    $\idim_{\Lambda_P} {}_{\Lambda_P}\Lambda_P=\dim \Lambda_P$ for
    any $P\in\Spec R$.
  \end{enumerate}
\end{corollary}

\begin{proof}
  {\bf 1$\Rightarrow$2} is trivial.

  {\bf 2$\Rightarrow$3}.
  By Proposition~\ref{GN.prop}, we have
  $\idim {}_{\Lambda_\fm}\Lambda_\fm=\dim \Lambda_\fm$
  for each $\fm$.
  Then by \cite[(4.7)]{GN}, $\Lambda$ is a Gorenstein $R$-algebra.

  {\bf 3$\Rightarrow$1} follows from Proposition~\ref{GN.prop}.
\end{proof}
  
\paragraph
Let $R$ be arbitrary.
We say that $\Lambda$ is a {\em quasi-Gorenstein} $R$-algebra if
$\Lambda_P$ is pseudo-quasi-Frobenius for each $P\in\Spec R$.

\begin{definition}[Scheja--Storch \cite{SS}]
Let $R$ be general.
We say that $\Lambda$ is symmetric (resp.\ Frobenius)
relative to
$R$ if $\Lambda$ is $R$-projective, and $\Lambda^*:=\Hom_R(\Lambda,R)$ is
isomorphic to $\Lambda$ as a $\Lambda$-bimodule (resp.\ as a right $\Lambda$-module).
It is called quasi-Frobenius relative to $R$
if the right $\Lambda$-module $\Lambda^*$ is projective.
\end{definition}

\begin{lemma}\label{SS.lem}
  Let $(R,\fm)$ be local.
  \begin{enumerate}
  \item[\bf 1] If $\dim \Lambda=\dim R$, $R$ is quasi-Gorenstein, and
    $\Lambda^* \cong \Lambda$ as $\Lambda$-bimodules
    \(resp.\ $\Lambda^*\cong\Lambda$ as right $\Lambda$-modules,
    $\Lambda^*$ is projective as a right $\Lambda$-module\), then
    $\Lambda$ is quasi-symmetric \(resp.\ pseudo-Frobenius, pseudo-quasi-Frobenius\).
  \item[\bf 2] If $R$ is Gorenstein and $\Lambda$ is symmetric \(resp.\
    Frobenius, quasi-Frobenius\) relative to $R$, then $\Lambda$ is symmetric
    \(resp.\ Frobenius, quasi-Frobenius\).
  \item[\bf 3] If $\Lambda$ is nonzero and $R$-projective, then
    $\Lambda$ is quasi-symmetric \(resp.\ pseudo-Frobenius,
    pseudo-quasi-Frobenius\) if and only if $R$ is quasi-Gorenstein and
    $\Lambda$ is symmetric \(resp.\ Frobenius, quasi-Frobenius\)
    relative to $R$.
  \item[\bf 4] If $\Lambda$ is nonzero and $R$-projective, then
    $\Lambda$ is symmetric \(resp.\ Frobenius, quasi-Frobenius\) if and only if
    $R$ is Gorenstein and $\Lambda$ is symmetric \(resp.\ Frobenius, quasi-Frobenius\)
    relative to $R$.
  \end{enumerate}
\end{lemma}

\begin{proof}
  We can take the completion, and we may assume that $R$ is complete local.
  
  {\bf 1}.
  Let $d=\dim \Lambda=\dim R$, and let $\Bbb I$ be the normalized dualizing complex
  (see \cite[(5.2)]{Hashimoto2}) of $R$.
  Then
  \[
  K_\Lambda= \Ext^{-d}_R(\Lambda,\Bbb I)\cong \Hom_R(\Lambda,H^{-d}(\Bbb I))
  \cong \Hom(\Lambda,K_R)\cong \Hom(\Lambda,R)=\Lambda^*
  \]
  as $\Lambda$-bimodules, and the result follows.

  {\bf 2}.
  We may assume that $\Lambda$ is nonzero.
  As $R$ is Cohen--Macaulay and $\Lambda$ is a finite projective $R$-module,
  $\Lambda$ is a maximal Cohen--Macaulay $R$-module.
  By {\bf 1}, the result follows.

  {\bf 3}.
  The \lq if' part follows from {\bf 1}.
  We prove the \lq only if' part.
  As $\Lambda$ is $R$-projective and nonzero, $\dim\Lambda=\dim R$.
  As $\Lambda$ is $R$-finite free, $K_\Lambda\cong \Hom_R(\Lambda,K_R)\cong
  \Lambda^*\otimes_R K_R$.
  As $K_\Lambda$ is $R$-free and $\Lambda^*\otimes_R K_R$ is nonzero and is isomorphic to
  a direct sum of copies of $K_R$, we have that $K_R$ is $R$-projective,
  and hence $R$ is quasi-Gorenstein, and $K_R\cong R$.
  Hence $K_\Lambda\cong \Lambda^*$, and the result follows.

  {\bf 4} follows from {\bf 3} easily.
\end{proof}

\paragraph
Let $(R,\fm)$ be semilocal.
Let a finite group $G$ act on $\Lambda$ by $R$-algebra automorphisms.
Let $\Omega=\Lambda*G$, the twisted group algebra.
That is, $\Omega=\Lambda\otimes_R RG=\bigoplus_{g\in G}\Lambda g$
as an $R$-module, and
the product of $\Omega$ is given by $(ag)(a'g')=(a(ga'))(gg')$
for $a,a'\in\Lambda$ and $g,g'\in G$.
This makes $\Omega$ a module-finite $R$-algebra.

\paragraph
We simply call an $RG$-module a $G$-module.
We say that $M$ is a $(G,\Lambda)$-module if $M$ is a $G$-module,
$\Lambda$-module, the $R$-module structures coming from that of the
$G$-module structure and the $\Lambda$-module structure agree, and
$g(am)=(ga)(gm)$ for $g\in G$, $a\in \Lambda$, and $m\in M$.
A $(G,\Lambda)$-module and an $\Omega$-module are one and the same thing.

\paragraph
By the action $((a\otimes a')g)a_1=a(ga_1)a'$,
we have that $\Lambda$ is a $(\Lambda\otimes\Lambda\op)*G$-module
in a natural way.
So it is an $\Omega$-module by the action $(ag)a_1=a(ga_1)$.
It is also a right $\Omega$-module by the action $a_1(ag)=g^{-1}(a_1a)$.
If the action of $G$ on $\Lambda$ is trivial, then these actions
make an $\Omega$-bimodule.

\paragraph
Given an $\Omega$-module $M$ and an $RG$-module $V$,
$M\otimes_R V$ is an $\Omega$-module by
$(ag)(m\otimes v)=(ag)m\otimes gv$.
$\Hom_R(M,V)$ is a right $\Omega$-module by
$(\varphi(ag))(m)= g^{-1}\varphi(a(gm))$.
It is easy to see that the standard isomorphism
\[
\Hom_R(M\otimes_R V,W)\rightarrow \Hom_R(M,\Hom_R(V,W))
\]
is an isomorphism of right $\Omega$-modules for a left $\Omega$-module
$M$ and $G$-modules $V$ and $W$.

\paragraph
Now consider the case $\Lambda=R$.
Then the pairing $\phi: RG\otimes_R RG \rightarrow R$ given by
$\phi(g\otimes g')=\delta_{gg',e}$ (Kronecker's delta)
is non-degenerate, and induces
an $RG$-bimodule isomorphism $\Omega=RG\rightarrow (RG)^*=\Omega^*$.
As $\Omega=RG$ is a finite free $R$-module, we have that
$\Omega=RG$ is symmetric relative to $R$.

\begin{lemma}\label{quasi-symmetric.lem}
If $\Lambda$ is quasi-symmetric \(resp.\ symmetric\) and the
action of $G$ on $\Lambda$ is trivial,
then $\Omega$ is quasi-symmetric \(resp.\ symmetric\).
\end{lemma}

\begin{proof}
  Taking the completion, we may assume that $R$ is complete.
  Then replacing $R$ by a Noether normalization of $R/\ann_R \Lambda$,
  we may assume that $R$ is a regular local ring, and
  $\Lambda$ is a faithful $R$-module.
  As the action of $G$ on $\Lambda$ is trivial, 
  $\Omega=\Lambda\otimes_R RG$ is quasi-symmetric (resp.\ symmetric),
  as can be seen easily.
\end{proof}

\paragraph
In particular, if $\Lambda$ is commutative quasi-Gorenstein
(resp.\ Gorenstein) and the action of $G$ on $\Lambda$ is
trivial, then
$\Omega=\Lambda G$ is quasi-symmetric (resp.\ symmetric).

\paragraph
In general, ${}_\Omega\Omega\cong \Lambda\otimes_R RG$ as $\Omega$-modules.

\begin{lemma}\label{isomorphic-Lambda-Omega.lem}
  Let $M$ and $N$ be right $\Omega$-modules, and let 
  $\varphi:M\rightarrow N$ be a homomorphism of right $\Lambda$-modules.
  Then $\psi:M\otimes RG\rightarrow N\otimes RG$ given by
  $\psi(m\otimes g)=g(\varphi(g^{-1}m))\otimes g$ is an $\Omega$-homomorphism.
  In particular,
  \begin{enumerate}
  \item[\bf 1] If $\varphi$ is a $\Lambda$-isomorphism, then
    $\psi$ is an $\Omega$-isomorphism.
  \item[\bf 2] If $\varphi$ is a split monomorphism in $\md \Lambda$,
    then $\psi$ is a split monomorphism in $\md\Omega$.
  \end{enumerate}
\end{lemma}

\begin{proof}
Straightforward.
\end{proof}

\begin{proposition}
  Let $G$ be a finite group acting on $\Lambda$.
  Set $\Omega:=\Lambda*G$.
  \begin{enumerate}
  \item[\bf 1] If the action of $G$ on $\Lambda$ is trivial and
    $\Lambda$ is quasi-symmetric \(resp.\ symmetric\), then so is $\Omega$.
  \item[\bf 2] If $\Lambda$ is pseudo-Frobenius \(resp.\ Frobenius\),
    then so is $\Omega$.
  \item[\bf 3] If $\Lambda$ is pseudo-quasi-Frobenius \(resp.\
    quasi-Frobenius\), then so is $\Omega$.
  \end{enumerate}
\end{proposition}

\begin{proof}
  {\bf 1} is Lemma~\ref{quasi-symmetric.lem}.
  To prove {\bf 2} and {\bf 3}, we may assume that $(R,\fm)$ is
  complete regular local and $\Lambda$ is a faithful module.

  {\bf 2}.
  \[
  (K_\Omega)_\Omega
  \cong
  \Hom_R(\Lambda\otimes_R RG,R)\cong \Hom_R(\Lambda,R)\otimes(RG)^*
  \cong
  K_\Lambda\otimes RG
  \]
  as right $\Omega$-modules.
  It is isomorphic to $\Lambda_\Omega\otimes RG\cong \Omega_\Omega$
  by Lemma~\ref{isomorphic-Lambda-Omega.lem}, {\bf 1}, since
  $K_\Lambda\cong \Lambda$ in $\md\Lambda$.
  Hence $\Omega$ is pseudo-Frobenius.
  If, in addition, $\Lambda$ is Cohen--Macaulay, then $\Omega$ is also
  Cohen--Macaulay, and hence $\Omega$ is Frobenius.

  {\bf 3} is proved similarly, using
  Lemma~\ref{isomorphic-Lambda-Omega.lem}, {\bf 2}.
\end{proof}

Note that the assertions for Frobenius and quasi-Frobenius properties also
follow easily from Lemma~\ref{SS.lem} and \cite[(3.2)]{SS}.

\section{Codimension-two argument}\label{codim-two.sec}

\paragraph
This section is the second part of this paper.
In this section,
we show that the codimension-two argument using the
existence of $2$-canonical modules in \cite{Hashimoto12} is still valid in
non-commutative settings, as announced in (\ref{second-part.par}).

\paragraph
Let $X$ be a locally Noetherian scheme, $U$ its open subscheme,
and $\Lambda$ a coherent $\O_X$-algebra.
Let $i:U\hookrightarrow X$ be the inclusion.

\paragraph
Let $\M\in\md\Lambda$.
That is, $\M$ is a coherent right $\Lambda$-module.
Then by restriction, $i^*\M\in \md i^*\Lambda$.

\paragraph
  For a quasi-coherent $i^*\Lambda$-module $\N$, we have an action
  \[
  i_*\N\otimes_{\O_X}\Lambda
  \xrightarrow{1\otimes u} i_*\N \otimes_{\O_X}i_*i^*\Lambda
  \rightarrow i_*(\N\otimes_{\O_U} i^*\Lambda)\xrightarrow{a} i_*\N,
  \]
  where $u$ is the unit map for the adjoint pair $(i^*,i_*)$.
  So we get a functor $i_*:\Mod i^*\Lambda \rightarrow \Mod \Lambda$, where
  $\Mod i^*\Lambda$ (resp.\ $\Mod \Lambda$) denote the category of quasi-coherent
  $i^*\Lambda$-modules (resp.\ $\Lambda$-modules).

\begin{lemma}\label{S_2-isom.lem}
  Let the notation be as above.
  Assume that $U$ is large in $X$ \(that is, $\codim_X(X\setminus U)\geq 2$\).
  If $\M\in (S_2')^{\Lambda\op,X}$, then the canonical map
    $u:\M\rightarrow i_*i^*\M$ is an isomorphism.
\end{lemma}

\begin{proof}
  Follows immediately from \cite[(7.31)]{Hashimoto12}.
\end{proof}

\begin{proposition}\label{codim-two.prop}
  Let the notation be as above, and let $U$ be large in $X$.
  Assume that there is a $2$-canonical right $\Lambda$-module.
  Then we have the following.
  \begin{enumerate}
  \item[\bf 1]
        If $\N\in (S_2')^{i^*\Lambda\op,U}$, then $i_*\N\in (S_2')^{\Lambda\op,X}$.
  \item[\bf 2] $i^*:(S_2')^{\Lambda\op,X}\rightarrow (S_2')^{i^*\Lambda\op,U}$ and
    $i_*:(S_2')^{i^*\Lambda\op,U}\rightarrow (S_2')^{\Lambda\op,X}$ are quasi-inverse
    each other.
  \end{enumerate}
\end{proposition}

\begin{proof}
  The question is local, and we may assume that $X$ is affine.

  {\bf 1}.
  There is a coherent subsheaf $\Q$ of $i_*\N$ such that $i^*\Q = i^*i_*\N=\N$
  by \cite[Exercise~II.5.15]{Hartshorne}.
  Let $\V$ be the $\Lambda$-submodule of $i_*\N$ generated by $\Q$.
  That is, the image of the composite
  \[
  \Q\otimes_{\O_X}\Lambda\rightarrow i_*\N\otimes_{\O_X}\Lambda\rightarrow i_*\N.
  \]
  Note that $\V$ is coherent, and $i^*\Q\subset i^*\V \subset i^*i_*\N=i^*\Q=\N$.

  Let $\C$ be a $2$-canonical right $\Lambda$-module.
  Let $?\da:=\uHom\Lop(?,\C)$, $\Gamma=\uEnd_{\Lambda}\C$,
  and
  $?\dda:=\uHom_{\Gamma}(?,\C)$.
  Let $\M$ be the double dual $\V\ddd$.
  Then $\M\in (S_2')^{\Lambda\op,X}$, and hence
  \[
  \M\cong i_*i^*\M\cong i_*i^*(\V\ddd)\cong i_*(i^*\V)\ddd\cong i_*(\N\ddd)\cong i_*\N.
  \]
  So $i_*\N\cong \M$ lies in $(S_2')^{\Lambda\op,X}$.

  {\bf 2} follows from {\bf 1} and Lemma~\ref{S_2-isom.lem} immediately.
\end{proof}

\end{document}